\documentclass{article}
\pdfoutput=1
\usepackage{amssymb}
\usepackage{amsmath}
\usepackage{amsthm}

\newtheorem{thm}{Theorem}[section]
\newtheorem{lem}[thm]{Lemma}

\newtheorem{cor}[thm]{Corollary}

\newtheorem{mydef}[thm]{Definition}
\numberwithin{equation}{section}
\author{Michael Music}

\title{The nonlinear Fourier transform for two-dimensional subcritical potentials}

\begin{document}
\maketitle
\begin{abstract}
The inverse scattering method for the Novikov-Veselov equation is studied for a larger class of Schr\"odinger potentials than could be handled previously. Previous work concerns so-called conductivity type potentials, which have a bounded positive solution at zero energy and are a nowhere dense set of potentials. We relax this assumption to include logarithmically growing positive solutions at zero energy. These potentials are stable under perturbations. Assuming only that the potential is subcritical and has two weak derivatives in a weighted Sobolev space, we prove that the associated scattering transform can be inverted, and the original potential is recovered from the scattering data. 
\end{abstract}


\section{Introduction}
Given a potential $q\in L^p(\mathbb R^2)$ with $1<p<2$, we study the scattering transform for the two-dimensional Schr\"odinger equation. We recall Faddeev's complex geometric optics solutions at zero energy \cite{faddeev}, defined as the solution $\psi(x,k)$ of the equation
\begin{equation}
\begin{cases}
(-\Delta+q(\cdot ))\psi(\cdot,k)=0\\
e^{-ikx}\psi(x,k)-1\in W^{1,\tilde p}(\mathbb R^2).
\end{cases}
\label{eq:schro}
\end{equation} \noindent
Here $x\in \mathbb R^2$ is identified with the complex number $x=x_1+ix_2$, and $k\in \mathbb C\setminus \{0\}$ is a parameter. The exponent $ \tilde p$ is defined by $1/\tilde p=1/p-1/2$. Such a solution is not guaranteed to be well-defined. If the Schr\"odinger equation in \eqref{eq:schro} has a solution with $e^{-ikx}h(x)\in W^{1,\tilde p}(\mathbb R^2)$ we call $k$ an exceptional point. 

If $k$ is not an exceptional point and $q$ is integrable, the scattering transform $\mathcal T:q\to {\bf t}$ is given by
\begin{equation}
{\bf t}(k)=\int_{\mathbb R^2}e^{i\bar k \bar x} q(x)\psi(x,k)dx.
\label{eq:scat}
\end{equation}

We will show the scattering transform is well-defined for a certain class of potentials for every $k\in \mathbb C\setminus \{0\}$. We will then show that the renormalized solutions $\mu(x,k)=e^{-ikx}\psi(x,k)$ satisfy the $\bar\partial_k$ equation
\begin{equation}
\begin{cases}
\bar \partial_k \mu(x,k)=e_{-x}(k)\displaystyle\frac{{\bf t}(k)}{4\pi\bar k}\overline{\mu(x,k)}\\
\\
\mu(x,\cdot)-1 \in L^{r}(\mathbb C)
\label{eq:dbar-k-mu}
\end{cases}
\end{equation}
where we define
\[ \bar\partial_k=\frac{1}{2}\left(\frac{\partial}{\partial k_1}+i\frac{\partial}{\partial k_2}\right),\]
\[e_{-x}(k)=\exp(-i(kx+\bar k \bar x)).\]
and $p'<r<\infty$.
Theorem \ref{thm:tk-small} and large-$k$ estimates on ${\bf t}(k)$ imply that ${\bf t}(k)/(4\pi\bar k) \in L^2(\mathbb C)$. This allows us to use uniqueness theorems for equations of the form \eqref{eq:dbar-k-mu}, so $\mu(x,k)$ can be recovered from the the data ${\bf t}(k)$ by solving that equation. This equation and the large-$k$ limits of its solutions are the basis for recovering $q(x)$ via the inverse scattering transform defined by 
\begin{equation}
\mathcal Q{\bf t}(x)=\frac{4i}{\pi}\bar\partial_x\int_{\mathbb C}\frac{{\bf t}(k)}{4\pi\bar k}e_{-x}(k)\overline{\mu(x,k)}dk,
\label{eq:inv-trans}
\end{equation}
where $dk$ is the Lebesgue measure.

The main interest we have in the inverse scattering transform is eventually to use it to solve the (2+1) dimensional Novikov-Veselov equation:
\begin{align*}
q_t&=-\bar\partial^3_x q -\partial^3_x q+\frac{3}{4}\bar\partial_x(q\bar v)+\frac{3}{4}\partial_x(qv)\\
v&=\bar\partial^{-1}_x\partial_x q.
\end{align*}
Formally, 
\[q(x,t)=\mathcal Q (e^{it(k^3+\bar k^3)}\mathcal T\left[q(\cdot, 0)\right])(x,t).\]
The reader can find a review of the general methods in two-dimensional inverse scattering in Beals and Coiffman \cite{beals}.  The inverse scattering method for the Novikov-Veselov equation was formulated by Boiti, Leon, Manna, and Pempinelli in \cite{boitietal}. We would direct the reader to the article by Croke, Mueller, Music, Perry, Siltanen, and Stahel \cite{croke} for an overview and other references for the inverse scattering method in the Novikov-Veselov equation. 
We hope this paper's analysis of the two transforms leads to a proof showing that the inverse scattering solution is a classical solution to the Novikov-Veselov equation for subcritical potentials.

The inverse scattering method for solving the Novikov-Veselov equation was previously studied by Tsai \cite{tsai} using small data assumptions to recover $q$ from the scattering data. In \cite{tsai2}, he gave a formal derivation of the evolution equations associated with the Novikov-Veselov equation. A breakthrough in the area came from work on the inverse conductivity problem in 1996 by Nachman \cite{nachman}. A potential, $q$ is called ``conductivity type'' if $q=\gamma^{-1/2}\Delta \gamma^{1/2}$ where $\gamma\in L^\infty(\mathbb R^2)$ and $\gamma>c>0$. Define $L^p_\rho(\mathbb R^2)=\{f:\langle x\rangle^{\rho}f(x)\in L^p(\mathbb R^2)\}$. Nachman \cite[Theorem 3]{nachman} proved the following:

\begin{thm}[Nachman, 1996]
Let $q\in L^p_\rho(\mathbb R^2)$ for $1<p<2$ and $\rho>1$. The following are equivalent:
\begin{enumerate}
\item $q$ is conductivity type
\item the scattering transform has no exceptional points and $|t(k)|\leq c|k|^\epsilon$ for some $\epsilon>0$ and $k$ sufficiently small.
\end{enumerate}
\label{thm:nachman}
\end{thm}

In the work of Lassas, Mueller, and Siltanen \cite{lms}, this result was used to attack the inverse scattering problem without small data assumptions  They prove regularity results for the direct and inverse scattering transforms, and they prove the identity $\mathcal Q[\mathcal T q]=q$ for a class of conductivity type potentials. Later, Lassas, Mueller, Siltanen, and Stahel \cite{lmss} study the inverse scattering evolution of radially symmetric, compactly supported, smooth initial data. In particular, they show that the evolution is well-defined and preserves conductivity type. Using the Miura map, Perry \cite{perry} proved that the inverse scattering method yields global solutions to the Novikov-Veselov equation for conductivity type initial data.

We are interested in extending the result of Nachman to a larger class of Schr\"odinger potentials called ``subcritical'' by Murata in \cite{murata}. Murata groups the set of all $L^p_{\mathrm{loc}}(\mathbb R^2)$ potentials into three categories:
\begin{mydef} 
A potential $q\in L^p_{\mathrm{loc}}(\mathbb R^2)$ is
\begin{itemize}
\item[(i)] \emph{subcritical} if there is a positive Green's function for the operator $-\Delta+q$,
\item[(ii)]\emph{critical} if there is no positive Green's function but $-\Delta+q\geq 0$ , and
\item[(iii)] \emph{supercritical} if $-\Delta+q\ngeq 0$.
\end{itemize}
\label{def:murata}
\end{mydef}
 For $q\in L^p_\rho(\mathbb R^2)$ with $1<p<2$, $\rho>2/p'$ where $p'$ is the conjugate exponent $1=1/p+1/p'$, Murata \cite[Theorem 5.6]{murata} shows that critical potentials are equivalent to the conductivity type potentials of Nachman.

There are two theorems from Murata's paper that help us understand the space of subcritical potentials.  Theorem 2.5 from Murata's paper says that if $w\in L^p_{\mathrm{loc}}(\mathbb R^2)$ is a nonnegative function not identically zero and $q$ is critical, then $q+w$ is subcritical. If instead $q$ is subcritical, then $q+w$ is also subcritical. This result tells us that the set of critical potentials is small compared to the set of subcritical potentials. Theorem 5.6 of the same paper proves that if $q\in L^p_\rho(\mathbb R^2)$ is subcritical, that there is a unique positive solution $\psi_0$ to the Schr\"odinger equation and the solution has large-$x$ asymptotics
\begin{equation}
\psi_0(x)=a\log|x|+ O(1)
\label{eq:subcrit-asymp}
\end{equation}
where $a>0$. The difference between the critical and subcritical cases then comes down to whether there is a bounded positive solution or a positive solution with logarithmic growth.

To prove the absence of exceptional points for critical potentials, Nachman uses the bounded positive solution to decompose the Schr\"odinger operator into first-order operators. A Liouville theorem for pseudoanalytic functions shows that there are no exceptional points. We do the same in this paper to prove the following:
\begin{thm}
A subcritical potential $q\in L^p_\rho(\mathbb R^2)$ with $1<p<2$ and $\rho>2/p'$ has no exceptional points.
\label{thm:no-exc}
\end{thm} \noindent
In order to do this, we prove a new Liouville theorem in Lemma \ref{lem:br-uhl-ext} which is a modified version of Theorem 3.1 from Brown and Uhlmann \cite{brown_uhlmann}.

We know from Nachman \cite[Theorem 3]{nachman} that the only potentials which have no exceptional points and with scattering transforms that behave like $c|k|^{\epsilon}$ for small $k$ are critical. Therefore, the scattering transform for subcritical potentials must have more singular behavior . In fact, we prove:
\begin{thm}
For a subcritical potential $q\in L^p_\rho(\mathbb R^2)$ for $1<p<2$ and $\rho>1$, the scattering transform satisfies
\begin{equation}
{\bf t}(k)=\frac{2\pi a}{c_\infty-a(\log|k|+\gamma)} +O(|k|^\epsilon)
\label{eq:tk-small}
\end{equation}
where $a>0$, $c_\infty \in \mathbb R$, and $|k|$ sufficiently small.
\label{thm:tk-small}
\end{thm} \noindent
The $a$ in this theorem is the same as that in equation \eqref{eq:subcrit-asymp}, and the $c_\infty$ is another constant related to $\psi_0$. This result was proved for a small class of radial potentials in \cite{exceptional_circle} and for point potentials in \cite{grinevich-novikov}.

We will then prove that the scattering transform and the inverse scattering transform are true inverses of each other on the space of subcritical potentials.  Lassas, Mueller, and Siltanen \cite{lms} proved $\mathcal Q\mathcal [T q](x)=q(x)$ for conductivity type potentials of the form $q=\gamma^{-1/2}\Delta \gamma^{1/2}$ where $\gamma-1\in C^{\infty}_0(\mathbb R^2)$. We weaken the regularity required and extend their result to subcritical potentials to prove:
\begin{thm}
For a critical or subcritical potential $q\in W^{2,p}_\rho(\mathbb R^2)$ with $1<p<2$ and $\rho>1$, we have 
\[\mathcal Q[\mathcal T q](x)=q(x).\]
\label{thm:inv-scat}
\end{thm}
For the potentials Lassas, Mueller, and Siltanen consider, ${\bf t}(k)$ is Schwartz class. They use this regularity in their proofs, but because of the small-$k$ behavior of our scattering transform we will need to use different methods. We prove in Lemma \ref{lem:sk-decay} that ${\bf t}(k)$ will have more decay as $|k|\to \infty$ when the potential is regular. The invertibility of the scattering transform then follows from analyzing the large-$k$ behavior of $\mu(x,k)$. 

The paper is organized as follows. In section 2, we introduce some previous known results that will be used throughout the paper. Section 3 proves Theorem \ref{thm:no-exc}, the absence of exceptional points for subcritical potentials. In section 4, we prove Theorem \ref{thm:tk-small}, the behavior of ${\bf t}(k)$ for small $k$. The final section is devoted to the large-$k$ behavior of ${\bf t}(k)$ and a study of the inverse problem \eqref{eq:dbar-k-mu}. After deriving a large-$k$ expansion of $\mu(x,k)$ in Lemma \ref{lem:large-mu-exp}, we use it to prove Theorem \ref{thm:inv-scat}.

\section{Preliminaries}
Here, we collect some of the results that are already known which we will be using throughout the remainder of the paper. The following Liouville-type theorem from Brown and Uhlmann \cite[Corollary~3.11]{brown_uhlmann} will be useful in many results throughout the paper. We will extend this result to a slightly larger class of functions in Corollary \ref{cor:liouville}. 
\begin{cor}
\cite{brown_uhlmann} Suppose $u\in L^p(\mathbb R^2)\cap L^2_{\mathrm{loc}}(\mathbb R^2)$ for some $p$, $1\leq p<\infty$ and satisfies the equation
\begin{equation}
\bar\partial u = au+b\bar u
\label{eq:vekua}
\end{equation}
where $a$ and $b$ lie in $L^2(\mathbb R^2)$. Then $u\equiv 0$.
\label{cor:bu-liouville}
\end{cor}

In the paper, we draw heavily from earlier results of Nachman \cite{nachman}. In Theorem 1.1 of his paper, he proves that if $q\in L^p(\mathbb R^2)$ , $1<p<2$ has no exceptional points, then for $k\in \mathbb C\setminus\{0\}$ there is a unique solution $\psi(x,k)$ of equation \eqref{eq:schro} and
\begin{equation}
\|e^{-ikx}\psi(\cdot,k)-1\|_{W^{s,\tilde p}}\leq c|k|^{s-1}\|q\|_{L^p}
\label{eq:schro-decay}
\end{equation}
for $0\leq s\leq 1$ and $k$ sufficiently large. We can rearrange the Schr\"odinger equation \eqref{eq:schro} when we replace $\psi(x,k)$ with $e^{-ikx}\mu(x,k)$ to get 
\begin{equation}
\bar\partial(\partial+ik)\mu=\frac{q\mu}{4}.
\label{eq:dbar-z}
\end{equation} 
One particular tool for studying this $\bar\partial$-equation which we shall employ in Section 5 is Lemma 1.4 from Nachman \cite{nachman} which we restate below.
\begin{lem}
\cite{nachman} For any $f\in L^p(\mathbb R^2)$ and any $k\in \mathbb C\setminus\{0\}$, there is a unique weak solution $u\in L^{\tilde p}(\mathbb R^2)$ solving
\begin{equation}
\bar\partial(\partial+ik)u=f.
\label{eq:partialk-u}
\end{equation}
Furthermore, $u\in W^{1,\tilde p}$ and 
\begin{equation}
\|u\|_{W^{s,\tilde p}}\leq \frac{c}{|k|^{1-s}}\|f\|_p\qquad \mbox{for } |k|\geq \mbox{const.}>0, \,0\leq s\leq 1.
\label{eq:k-decay}
\end{equation}
\label{lem:k-decay}
\end{lem}

In Section 4, where we study the small-$k$ behavior of ${\bf t}(k)$, we will need various estimates on Faddeev's Green's function, $g_k$, given by
\begin{equation}
g_k(x)=\frac{1}{(2\pi)^2}\int_{\mathbb R^2}\frac{e^{ix\cdot \xi}}{|\xi|^2+2k(\xi_1+i\xi_2)}\,d\xi,
\label{eq:faddeev}
\end{equation}
here $x\cdot \xi = x_1\xi_1+x_2\xi_2$.
This is the convolution operator for equation \eqref{eq:partialk-u}. For $f\in L^p$, we may write $u=g_k\ast f$. To find out how $\mu$ behaves near $k=0$, we split off of $g_k(x)$ a logarithmic function $\ell(k)=(\log|k|+\gamma)/2\pi$ where $\gamma$ is the Euler constant. We then define 
\begin{equation}
\tilde g_k(x)=g_k(x)+\ell(k),
\label{eq:gk-tilde}
\end{equation} 
because $\tilde g_k$ does not have the singular small-$k$ behavior that $g_k$ does. Denoting $G_0(x)=-\log|x|/2\pi$, Nachman \cite[Lemma 3.4]{nachman} proves the estimate
\begin{equation}
|\tilde g_k(x)-G_0(x)|\leq C_\epsilon |k|^\epsilon\langle x\rangle^{\epsilon}
\label{eq:gktilde-est}
\end{equation}
for $0<\epsilon<1$ and for all $0<|k|\leq 1/2$. Along with this estimate on $g_k(x)$, we need the following two inequalities to help us with a uniqueness result in the proof of small-$k$ behavior of ${\bf t}(k)$. The next lemma is also from Nachman \cite[Lemma 3.4]{nachman}.
\begin{lem}
\cite{nachman} Let $f\in L^p_\rho$ for $1<p<2$ and $\rho>1$. Then
\begin{equation}
\left\|G_0\ast f+\frac{1}{2\pi}(\log|x|)\int_{\mathbb R^2}f\right\|_{L^{\tilde p}}\leq c\|f\|_{L^p_\rho}
\label{eq:g0f}
\end{equation}
and
\begin{equation}
\|\nabla G_0\ast f\|_{L^{\tilde p}} \leq c\|f\|_{L^p}.
\label{eq:dg0f}
\end{equation}
\label{lem:estimates}
\end{lem}
The inverse problem for the Novikov-Veselov equation relies on studying equation \eqref{eq:dbar-k-mu}. In order to prove Theorem \ref{thm:inv-scat}, we need these solutions to be the same as those from equation \eqref{eq:dbar-z}. Nachman proves $\mu(x,k)$ satisfy 
\[\bar\partial_k \mu(x,k)=e_{-x}(k)\frac{{\bf t}(k)}{4\pi\bar k}\overline{\mu(x,k)}\]
point-wise  in the topology $W^{1,\tilde p}_{-\beta}(\mathbb R^2)=\{f:\langle x \rangle^{-\beta} f \in W^{1,\tilde p}(\mathbb R^2)\}$ for any $k\in \mathbb C\setminus \{0\}$ which is not an exceptional point. Combining this with inequality \eqref{eq:muconv} and identity \eqref{eq:mutilde} shows that $\mu(x,\,\cdot\,)$ defined by \eqref{eq:dbar-z} is in $L^r(\mathbb C)$ for $r>p'$, and so it satisfies \eqref{eq:dbar-k-mu}. Therefore, the solutions defined by each equation are the same. 

To finish studying the properties of ${\bf t}(k)$ and the large-$k$ asymptotic behavior of $\mu$ we will will need various results on the operator $\bar\partial^{-1}$. Part 1 is the classical Hardy-Littlewood-Sobolev inequality which can be found in Astala, Iwaniec, and Martin section 4.3 \cite{astala}. A version of parts 2 and 3 can be found in Nachman \cite[Lemma 1.4]{nachman}, or a slightly more general version in \cite{astala}. Define $\bar\partial^{-1} f$ for functions in $L^p$, $1<p<2$, by
\[\bar\partial^{-1} f(x) = \int_{\mathbb R^2} \frac{f(y)}{x-y}\,dy,\]
and for functions in $L^r$ for $r>2$ by density.
\begin{lem}\cite{astala}\cite{nachman}
\begin{enumerate}
\item If $f\in L^p(\mathbb R^2)$ for $p<2$ then $\|\bar\partial^{-1} f\|_{L^{\tilde p}}\leq \|f\|_{L^p}$
\item If $f\in L^q(\mathbb R^2)$ for $q>2$ then $\bar\partial^{-1}f$ belongs to the space $C^{\alpha}(\mathbb R^2)$ for $\alpha=1-2/q$. 
\item If $f\in L^p\cap L^q(\mathbb R^2)$ for $1<p<2<q$, then 
\[\|\bar\partial^{-1}f\|_{L^\infty}\leq c_{q,p}(\|f\|_{L^{p}}+\|f\|_{L^q}),\]
and $\lim_{|x|\to \infty} u(x)=0$.
\end{enumerate}
\label{lem:dbar-estimates}
\end{lem}
\section{Absence of exceptional points}

We prove that the exceptional set is empty by showing that any solution to the Schr\"odinger equation with $e^{-ikx}h(x)\in W^{1,\tilde p}(\mathbb R^2)$ gives rise to an auxiliary function which solves an equation of the form \eqref{eq:vekua} with $L^2(\mathbb R^2)$ coefficients. This means we will need to use the Liouville theorem from Brown and Uhlmann \cite{brown_uhlmann}. One problem that the following Lemma and Corollary handle is that the auxiliary function, $v$, will not necessarily belong to any $L^p$ space. It will instead belong to a weighted space $L^p_{-\epsilon}(\mathbb R^2)$ for any $\epsilon>0$, because our positive solution to the Schr\"odinger equation, $\psi_0$, has logarithmic growth instead of being bounded. 

In the proof, we define $\bar\partial^{-1}$ for $L^2(\mathbb R^2)$ functions by
\[\bar\partial^{-1}f(x)=\frac{1}{\pi}\int_{\mathbb R^2}\left[\frac{1}{x-y}-\frac{\chi(y)}{y}\right]f(y)\,dy\]
where $\chi(y)\in C^{\infty}$ is zero for $|y|<1$ and equal to one in $|y|>2$. Note that this gives a particular representative for $\bar\partial^{-1}f$ in the class of functions of bounded mean oscillation.
With that in mind, we extend Brown and Uhlmann's Liouville theorem to include some negatively weighted spaces. 

\begin{lem}
Suppose $f$ is in $L^2(\mathbb{R}^2)$ and $w\in L^p_{-\rho}(\mathbb R^2)$ for $1\leq p<\infty$ and $\rho<\min(1,2/p)$, and assume that $w\exp(-\bar\partial^{-1}f)$ is holomorphic. Then $w$ is zero.
\label{lem:br-uhl-ext}
\end{lem}
\begin{proof}
Let $B_r(x)$ be the disk $\{y:|x-y|<r\}$. We denote the average of the function $v$ on a disk $B$ by
\[ v_B=\mu(B)^{-1}\int_B v\,dx.\]
Also let $u=-\bar\partial^{-1}f$. We collect here the estimates that Brown and Uhlmann prove for $\exp(u)$.
\begin{enumerate}
\item For $x,y\in \mathbb R^2$ with $r<s$, we have
\begin{equation}
\left|u_{B_r(x)}-u_{B_s(y)}\right|\leq C\|f\|_{L^2}\left(\log(|x-y|/s+s/r+2)\right)^{1/2}.
\label{eq:u-averages}
\end{equation}
\item For given $p>1$, there exists $C>0$ and $r_0=r_0(p,f)$ so that if $r<r_0$ then
\begin{equation}
\int_{B_r(x)}\exp(p'|u-u_{B_r(x)}|)dx\leq C\mu(B_r(x)).
\label{eq:expavg}
\end{equation}
\item For each $\epsilon>0$ there exists an $R_0>0$ and $C>0$ such that the following inequality holds:
\begin{equation}
\int_{B_r(0)}\left|\exp(-u)\right|\,dx\geq \mu(B_r(0))^{1-\epsilon}\exp(-C\|f\|_{L^2}), \qquad r>R_0.
\label{eq:expu-increasing}
\end{equation}
\end{enumerate}
Equation \eqref{eq:u-averages} immediately implies that for fixed $r$,
\[u_{B_r(x)}= O((\log|x|)^{1/2}), \qquad \mbox{as } x\to \infty.\]
Taking the exponential of the function we find that for fixed $r$ and any $\epsilon>0$,
\begin{equation}
\exp(u_{B_r(x)})=o(|x|^\epsilon), \qquad \mbox{as } x\to \infty.
\label{eq:avg}
\end{equation}
Using \eqref{eq:expavg} and H\"older's inequality we find
\begin{align*}
|(e^uw)_B|&\leq |e^{u_B}|\mu(B)^{-1}\|\langle x \rangle^{\rho}\|_{L^\infty(B)}\|\|\exp(p'|u-u_B|)\|_{L^{p'}(B)}\|w\langle x\rangle^{-\rho} \|_{L^p(B)}\\
&\leq  |e^{u_B}|\mu(B)^{-1}\left(\sup_{x\in B}\langle x\rangle^{\rho}\right) \left(\int_B \exp(p'|u-u_B|)dx\right)^{1/p'}\|w\|_{L^p_{-\rho}}.
\end{align*}
Using \eqref{eq:avg} and \eqref{eq:expavg} we see that for fixed $r_0=r_0(p,u)$, $\rho<1$ and $\epsilon=1-\rho$
\[|(e^uw)_B|=o(|x|), \qquad \mbox{as } x\to \infty\]
Since $e^uw$ is assumed holomorphic, this implies that $e^uw$ is constant, so we have $w=C_0e^{-u}$.

We want to prove $C_0=0$. We have from H\"older's inequality
\begin{align*}
\int_{B_r(0)}\left|w\right|\,dx&\leq \left(\int_{B_r(0)}\langle x \rangle^{\rho p'}dx\right)^{1/p'}\|w\|_{L^p_{-\rho}}\\
&\leq 2r^\rho\mu(B_r(0))^{1/p'}\|w\|_{L^p_{-\rho}}\\
&\leq c\mu(B_r(0))^{1-1/p+\rho/2}\|w\|_{L^p_{-\rho}}.
\end{align*}
By assumption $1/p-\rho/2>0$, so choosing $0<\epsilon<1/p-\rho/2$ in \eqref{eq:expu-increasing} we see that the asymptotic growth of the two averages are different. The two inequalities then imply that $C_0=0$.
\end{proof}

This leads to a new Liouville theorem for pseudo-analytic functions (see Lemma \ref{cor:bu-liouville} for the original version by Brown and Uhlmann \cite{brown_uhlmann}):
\begin{cor}
If the function $v\in L^p_{-\rho}\cap L^2_{\mathrm{loc}}(\mathbb R^2)$ for $1\leq p<\infty$ and $\rho<\mbox{min}(1,2/p)$ solves
\[\bar\partial v = av+b\bar v\]
with coefficients $a, b\in L^2(\mathbb C)$ then $v\equiv 0$.
\label{cor:liouville}
\end{cor}
\begin{proof}
We define a new function $f$ by
\begin{equation*}
f=\left\{
\begin{array}{ll}
a+b \displaystyle\frac{\bar v}{v}& \mbox{if } v\neq 0\\
\\
0& \mbox{if } v=0.
\end{array}
\right.
\end{equation*}
Then $\bar\partial\left(v e^{-\bar\partial^{-1}f}\right)=0$. By Lemma \ref{lem:br-uhl-ext}, $v\equiv 0$.
\end{proof}

When we use equation \eqref{eq:vekua}, our coefficients are of the form $\bar \partial \psi_0/\psi_0$ and $\partial \psi_0/\psi_0$ where $\psi_0$ is the positive solution to the Schr\"odinger equation for the subcritical potential $q$. To use the Liouville theorem, we need these two functions to be in $L^2$. Define the weight $W(x)=\log(|x|+e)$.

\begin{lem}
If $q(x)\in L^p_\rho(\mathbb R^2)$ for $1<p<2$ and $\rho>2/p'$ then the functions $\bar\partial \psi_0(x)/W(x)$ and $\partial\psi_0(x)/W(x)$ are in $L^2(\mathbb R^2)$.
\label{lem:pos-largex}
\end{lem}
\begin{proof}
We write $\partial \psi_0(x)=f(x)+\frac{1}{4}\bar\partial^{-1}q\psi_0(x)$ for some analytic function $f(x)$. 

We will show that $f(x)\neq 0$ is incompatible with the growth of $\psi_0$. By the Hardy-Littlewood-Sobolev inequality, $\bar\partial^{-1}q\psi_0 \in L^r$ for all $\tilde p<r<\infty$. We will integrate against a radially-symmetric non-negative bump function, $g\in C^\infty(\mathbb R^2)$, which is $1$ for $|x|<1$ and $0$ for $|x|>2$. For fixed $x\in \mathbb R^2$ and large $R>0$ we have
\begin{multline*}
\left|\int g\left(\frac{|y-x|}{R}\right) \partial \psi_0(y) dy \right|\\
\geq \left|\int g\left(\frac{|y-x|}{R}\right) f(y) dx\right|-\left|\int g\left(\frac{|y-x|}{R}\right)\frac{1}{4}\bar\partial^{-1}q\psi_0(y) dy\right|.
\end{multline*}
Integrating by parts on the left and using the mean value property and H\"older's inequality on the right, we get
\[\left|\int (a\log|R|+O(1))\partial g\left(\frac{|y-x|}{R}\right) dy\right| \geq c_1R^2|f(x)|-c_2R^{2/r'}\|\bar\partial^{-1}q\psi_0\|_{L^r}.
\]
We have $\partial g(|x-y|/R)=O(1/R)$, so
\[c_3 R\geq c_1 R^2|f(x)|-c_2R^{2/r'} \|\bar\partial^{-1}q\psi_0\|_{L^r}\]
Taking $R$ large shows $f(x)=0$ for all $x$.

Using the equality $\partial \psi_0(x)=\frac{1}{4}\bar\partial^{-1}q\psi_0(x)$ we show $\bar\partial \psi_0/W \in L^2(\mathbb R^2)$. Let $f(x)=q\psi_0\in L^p_\rho(\mathbb R^2)$ for $1<p<2$ and $\rho>2/p'$. First we separate the $L^2$ norm into three regions.
\begin{align*}
\left\|\frac{\bar\partial^{-1}f(x)}{W(x)}\right\|_{L^2}^2\leq&\frac{1}{\pi}\int_{|x|<1}\frac{1}{W(x)^2}\left(\int_{\mathbb R^2} \frac{f(x)}{x-y}\,dy\right)^2\,dx\\
&+\frac{1}{\pi}\int_{|x|>1}\frac{1}{W(x)^2}\left(\int_{|x-y|<|x|/2}\frac{f(y)}{x-y}\,dy\right)^2\,dx\\
&+\frac{1}{\pi}\int_{|x|>1}\frac{1}{W(x)^2}\left(\int_{|x-y|>|x|/2}\frac{f(y)}{x-y}\,dy\right)^2\,dx\\
=&\mbox{I}+\mbox{II}+\mbox{III}.
\end{align*}
For the first two integrals we will not use the extra $W(x)$ weight. For integral I, we have since $f\in L^p(\mathbb R^2)$ that $\bar\partial^{-1}f\in L^{\tilde p}(\mathbb R^2)\subset L^2_{\mathrm{loc}}(\mathbb R^2)$ by the Hardy-Littlewood-Sobolev inequality. 

For integral II , we use H\"older's inequality and then the Hardy-Littlewood-Sobolev inequality on $\langle \cdot \rangle^{\rho}f\in L^p$ for $1<p<2$ to get
\begin{align*}
\mbox{II}&\leq c\int_{|x|>1}\frac{\langle x\rangle^{-2\rho}}{W(x)^2}\left(\int_{|x-y|<|x|/2}\frac{\langle y\rangle^{\rho}f(y)}{x-y}\,dy\right)^2\,dx\\
&\leq c\|\langle x\rangle^{-2\rho}\|_{L^{p'}}\|\bar\partial^{-1}\langle\cdot \rangle^{\rho}f(\cdot)\|^2_{L^{2p/(2-p)}}\\
&\leq c\|\langle x\rangle^{-2\rho}\|_{L^{p'}}\|\langle \cdot \rangle^{\rho}f(\cdot)\|^2_{L^{p)}}.
\end{align*}
For integral III, we simply use H\"older's inequality, the embedding $L^p_\rho(\mathbb R^2)\subset L^1(\mathbb R^2)$, and the extra $(\log|x|)^2$ weight to get
\begin{align*}
\mbox{III}\leq& \frac{1}{\pi}\int_{|x|>1}\frac{1}{W(x)^2}\left(\int_{|x-y|>|x|/2}\frac{f(y)}{x-y}\,dw\right)^2\,dx\\
\leq& c\int_{|x|>1}\frac{1}{|x|^2W(x)^2}\,dx\left(\int_{\mathbb C}|f(y)|\,dy\right)^2\\
\leq& c\|f\|_{L^1}^2
\end{align*}
The last because for $|x|>1$, the function $(|x|W(x))^{-2}\in L^1(\mathbb R^2)$.
\end{proof}

We are now ready to prove Theorem \ref{thm:no-exc}. This result should be compared to Nachman's Lemma 1.5 \cite{nachman}.

\begin{proof}[Proof of Theorem \ref{thm:no-exc}]
Assume that $k$ is an exceptional point for a subcritical $q\in L^p_{\rho}(\mathbb R^2)$. Then there is an $h\neq 0$ satisfying $(-\Delta+q)h=0$ and $he^{-ikx}\in W^{1,q}(\mathbb R^2)$ for some $2\leq q<\infty$. We will show that this implies $h\equiv 0$.

Assume without loss of generality that $h$ is real-valued. Define 
\[v=(\psi_0\partial h-h\partial \psi_0)e^{-ikx}\] 
where $\psi_0$ is a positive solution to $(-\Delta+q)\psi_0=0$. We have $\psi_0\partial h e^{-ikx}\in L^{q}_{-\epsilon}(\mathbb R^2)$ for every $\epsilon>0$. We also have $he^{-ikx}\in L^\infty(\mathbb R^2)$ giving us $h\partial\psi_0 e^{-ikx}\in L^q(\mathbb R^2)$ and so $v\in L^{q}_{-\epsilon}(\mathbb R^2)$. Using the fact that $h$ and $\psi$ are real valued, we find
\begin{equation}
\bar \partial v= (\bar \partial\psi_0/\psi_0)v-(e_{-k}\partial \psi_0/\psi_0)\bar v.
\label{eq:v-vekua}
\end{equation}
Lemma \ref{lem:pos-largex} together with the asymptotics \eqref{eq:subcrit-asymp} gives us the coefficients to $v$ and $\bar v$ are in $L^2(\mathbb R^2)$.  An application of \ref{cor:liouville} shows that $v\equiv 0$, and so $(h/\psi_0)e^{-ikx}$ is analytic and in $L^q(\mathbb R^2)$ forcing $h\equiv 0$.
\end{proof}

\section{Small-$k$ behavior of the scattering transform}

In this section, we prove Theorem \ref{thm:tk-small} which  describes the small-$k$ behavior of ${\bf t}(k)$. When taking into account the decay of ${\bf t}(k)$, this theorem implies that the coefficient ${\bf t}(k)/(4\pi \bar k)$ in equation \eqref{eq:dbar-k-mu} is in $L^2(\mathbb R^2)$. This will allow us to use the Brown and Uhlmann Liouville theorem \cite{brown_uhlmann} to show uniqueness in the inverse problem.

In Theorem \ref{thm:tk-small}, we will need to make reference to the positive solution, $\psi_0$ to the Schr\"odinger equation. When $q\in L^p_\rho$ with $1<p<2$ and $\rho>1$, Nachman proves $\psi_0$ solves the integral equation
\begin{equation}
\psi_0=c_\infty-G_0\ast (q\psi_0)
\label{eq:cinfty}
\end{equation}
for some real number $c_\infty$ \cite[Lemma 3.1]{nachman}. It therefore will become necessary to have $c_\infty\neq 0$ when inverting the operators in the proof. The following lemma will be used to handle the special case of $c_\infty=0$ in Theorem \ref{thm:tk-small}.
\begin{lem} 
For $q\in L^p_\rho(\mathbb R^2)$ for $1<p<2$ and $\rho>1$, let $\psi_0(x)$ be the positive solution to $(-\Delta+q)\psi_0=0$ which for large $x$ has the form $\psi(x)_0=a\log|x|+O(1)$. With $c_\infty \in \mathbb R$, the scaled potential scattering transform associated with the rescaled potential $q_r(x)=r^2q(rx)$ satisfies $\mathbf{t}_r(k)=\mathbf{t}(k/r)$ with  $\psi_0^r(x)=a\log|x|+O(1)$ and $c^r_\infty = c_\infty + a \log|r|$.
\label{lem:scale}
\end{lem}
\begin{proof}
That the scattering transform behaves this way under scaling is proved in Theorem 3.19 of \cite{siltanen}. From Nachman Lemma 3.1, by integrating $\psi_0+G_0\ast(q\psi_0)$ over a ball of radius $R$, we have
\begin{equation}
c_\infty \pi R^2=\int_{|x|<R} \psi_0(x)dx-\left(\frac{1}{2}R^2\log R+\frac{1}{4}R^2\right)\int_{\mathbb R^2}q\psi_0+O(R^{2/p'}).
\label{eq:psiconstants}
\end{equation}
With the scaling we can calculate $\psi_0^r(x)=\psi_0(rx)=a\log|x|+a\log|r|+O(1)$ for large $x$. Using this with the above identity gives us that $c^r_\infty = c_\infty+a\log|r|.$
\end{proof}

In the following proof, we need control over the small $k$ behavior of Faddeev's Green's function $g_k(x)$. The solutions $\mu(x,k)=e^{-ikx}\psi(x,k)$ for $k\in \mathbb C\setminus \{0\}$ satisfy the integral equation
\begin{equation}
\mu=1+g_k\ast (q\mu).
\label{eq:mu-int}
\end{equation}
Nachman \cite{nachman} studied this operator in his section 3. At this point, it becomes convenient to put weights on the operator $\tilde g_k\ast(q\, \cdot\,)$ defined in \eqref{eq:gk-tilde}. We define the convolution operator $\tilde K(k)f=\langle x\rangle^{-\beta}\tilde g_k\ast (\langle \,\cdot\, \rangle^\beta qf)$ for $k\neq 0$ and $\tilde K(0)f=\langle x\rangle^{-\beta} G_0\ast (\langle \cdot \rangle^\beta qf)$.

With $\tilde K(k)$ defined, Nachman \cite[Lemma 3.5]{nachman} proves that if $q\in L^p_\rho(\mathbb R^2)$ for $1<p<2$ and $\rho>1$ then $\tilde K(k)$ is bounded from $W^{1,\tilde p}\to W^{1,\tilde p}$ and
\begin{equation}
\|\tilde K(k)-\tilde K(0)\|_{W^{1,\tilde p}\to W^{1,\tilde p}} \leq c|k|^\epsilon \mbox{ for } |k|<1/2
\label{eq:small-operator}
\end{equation}
for $0<\epsilon<\min((\rho-1)/2,2/p')$ and $2/\tilde p+\epsilon<\beta<\min(1,\rho-\epsilon-2/p')$.

\begin{proof}[Proof of Theorem \ref{thm:tk-small}]
We first study the case when $c_\infty \neq 0$ in equation \eqref{eq:cinfty} for our unique logarithmically growing solution $\psi_0$. 

If we can show that $I+\tilde K(0)$ is invertible, then using \eqref{eq:small-operator}, the more general $I+\tilde K(k)$ will be invertible by estimate \eqref{eq:small-operator}. Assume $\tilde h\in W^{1,\tilde p}(\mathbb R^2)$ is in the kernel of $I+\tilde K(0)$. Then $h=\langle x\rangle^{\beta}\tilde  h \in W^{1,\tilde p}_{-\beta}(\mathbb R^2)$ satisfies $h=-G_0\ast (qh)$ and $-\Delta h+qh=0 $ in the sense of distributions. By Sobolev embedding, $h\in L^{\infty}_{-\beta}$, so by the Hardy-Little-Sobolev inequality and differentiating the integral equation, $\nabla h \in L^{\tilde p}$.

Using the same argument as in the proof of Theorem \ref{thm:no-exc} we have the auxiliary function
\[v=h\partial \psi_0-\psi_0\partial h.\]
By Lemma \ref{lem:pos-largex}, $\partial \psi_0/\log(|x|+e)\in L^{2}$ so $h\partial\psi_0\in L^{2}_{-\beta-\epsilon}$ for any $0<\epsilon<1-\beta$. Simply by the logarithmic growth of $\psi_0$, we have $\psi_0\in L^{p'}_{-2/p'-\epsilon'}$ for $0<\epsilon'<1-2/p'$. By H\"older's inequality the product $\psi_0\partial h\in L^2_{-2/p'-\epsilon}$. Combining the two results, we set $a=\max(\beta+\epsilon, 2/p'+\epsilon')<1$, and we get the function $v\in L^2_{-a}$ solves \eqref{eq:v-vekua} with $k=0$. By Corollary \ref{cor:liouville}, $v\equiv 0$. This implies that $(h/\psi_0)$ is analytic and in $L^{\infty}_{-\beta}$ with $\beta<1$ so is a constant, $c$, by the usual Liouville theorem. However, $h=-G_0\ast qh = -cG_0 \ast q\psi_0=c_\infty c - c\psi_0$. Since $c_\infty \neq 0$ by assumption, we must have $c=0$ and thus $h\equiv 0$.
 
For the next part of the proof, start by defining solutions
\[\tilde\mu(x,k)=\langle x\rangle^{\beta}(I+\tilde K(k))^{-1}\langle \cdot \rangle^{-\beta}.\]
 Using the operator bounds on $\tilde K$ in equation \eqref{eq:small-operator} shows that
\begin{equation}
\|\tilde \mu(\cdot,k)-\psi_0/c_{\infty}\|_{W^{1,\tilde p}_{-\beta}}\leq c|k|^\epsilon
\label{eq:muconv}
\end{equation}
 for small $k$. Since $I+K(k)$ is invertible on $W^{1,\tilde p}$ for $k\neq 0$, we use the resolvent equation to give
\[(I+K(k))^{-1}=(I+\tilde K(k))^{-1}+(I+\tilde K(k))^{-1}[\tilde K(k)-K(k)](I+K(k))^{-1} \]
Here $[\tilde K(k)-K(k)]f=(\log|k|+\gamma)/2\pi\ast (qf)$. Applying both sides of the operator to $\langle x\rangle^{-\beta}$ we find
\begin{equation}
\mu(x,k)=(1+\ell(k)\tau(k))\tilde \mu(x,k)
\label{eq:mutilde}
\end{equation}
where
\[\tau(k):=\int_{\mathbb R^2}q(x)\mu(x,k)dx.\]
Integrating equation \eqref{eq:mutilde} against $q$ gives 
\[\tau(k)=(1+\ell(k)\tau(k))\tilde \tau(k)\]
 and solving for $\tau(k)$ gives us 
\[\tau(k)=\frac{\tilde \tau(k)}{1-\ell(k)\tilde\tau(k)}.\] Using the small-$k$ estimate for $\tilde\mu(x,k)$, we have
\[|\tilde \tau(k)|\leq \left|\int q(\tilde\mu(x,k)-\psi_0/c_\infty) dx\right|+\left|\int q\psi_0/c_\infty dx\right| \leq 2\pi a/c_\infty +c|k|^{\epsilon},\]
where we use the fact that $\int_{\mathbb R^2}q\psi_0 = 2\pi a$. Here $a$ is the coefficient to the logarithm in equation \eqref{eq:subcrit-asymp}, which can be obtained from equation \eqref{eq:psiconstants}.
Therefore, for small $k$,
\[\tau(k)=\frac{2\pi a/c_\infty+O(k^\epsilon)}{1-(\log|k|+\gamma) a/c_\infty+O(k^{\epsilon})}=\frac{2\pi a}{c_\infty-(\log|k|+\gamma) a}+O(k^\epsilon).\]
To finish the case when $c_\infty\neq 0$, note that for small $k$, we have $\tau(k)=\mathbf{t}(k)+O(k)$.

Now we consider the case when $c_\infty = 0$. Combining the result for $c_\infty\neq 0$ with the scaling properties from Lemma \ref{lem:scale}, the asymptotics for an arbitrarily scaled potential $q_r(x)$ are
\[\mathbf{t}_r(k)=\frac{2\pi a}{c^r_\infty-(\log|k|+\gamma) a}+O(k^\epsilon).\]
Using the relation $\mathbf{t}(k)=\mathbf{t}_r(rk)$ we get
\begin{align*}
\mathbf{t}(k)&=\frac{2\pi a}{c^r_\infty-(\log|rk|+\gamma) a}+O(k^\epsilon)\\
&=\frac{2\pi a}{(c^r_\infty-a\log|r|)-(\log|k|+\gamma) a}+O(k^\epsilon)\\&=\frac{2\pi a}{-(\log|k|+\gamma) a}+O(k^\epsilon).
\end{align*}
\end{proof}


\section{Large-$k$ asymptotics of the scattering transform and $\mu(x,k)$}

In order to prove Theorem \ref{thm:inv-scat}, we show that when $q$ is $n$ times weakly differentiable $\mu(x,k)$ will be $(n+1)$ times differentiable in $x$ and $(\mu(x,k)-1)$ will vanish in norm as $|k|\rightarrow \infty$. We then use this to show that $|k|^n{\bf t}(k)/\bar k\in L^r(|k|>k_0)$ for $r>\tilde p'$. With these two facts in hand, equation \eqref{eq:dbar-k-mu} implies that $\mu(x,k)$ has a large-$k$ expansion. The derivative of the first term in this expansion is $\mathcal Q$. Taking the $L^r(\mathbb C)$ limit as $|k|\to \infty$ and using equation \eqref{eq:dbar-z} then proves Theorem \ref{thm:inv-scat}. Siltanen proved a variation on the following lemmas on the decay of $\mu$ and ${\bf t}$ for compactly supported conductivity type potentials in section 3.2.1 of \cite{siltanen}.

\begin{lem} 
If $q\in W^{n,p}(\mathbb R^2)$ for $1<p<2$ has no exceptional points then $\mu(\cdot,k)-1\in W^{n+1,\tilde p}(\mathbb R^2)$ and 
\begin{equation}
\|D^{\alpha}(\mu(\cdot,k)-1)\|_{W^{s,\tilde p}(\mathbb R^2)}\leq \frac{c}{|k|^{1-s}}\sum_{m=0}^n\|q\|_{W^{m,p}(\mathbb R^2)} \|q\|_{L^2(\mathbb R^2)}^{n-m}
\label{eq:mudiff}
\end{equation}
for $0\leq s\leq 1$, $|\alpha|\leq n$, $c(n)$, and $k>k(\|q\|_{W^{n,p}},n)$.
\label{lem:mu-diff-forward}
\end{lem}
\begin{proof}
The case $n=0$ is Nachman's Theorem 1.1. We induct on this using Lemma \ref{lem:k-decay}. Take derivatives in equation \eqref{eq:dbar-z} with $u=(\mu-1)$, $f=q(\mu-1)/4+q/4$, and assume the result for all multi-indices less than $\alpha$. We have $D^{\alpha}(\mu-1)\in L^{\tilde p}(\mathbb R^2)$ solves equation \eqref{eq:dbar-z} with 
\begin{align}
f&=\frac{1}{4}\left[D^\alpha q+qD^{\alpha}(\mu-1)+\sum_{\beta:0<\beta< \alpha} \binom{\alpha}{\beta}D^{\beta} q D^{\alpha-\beta} (\mu-1)\right]\\
&=\frac{1}{4}(I+II+III)
\label{eq:threeparts}
\end{align}
We estimate the asymptotic behavior using the three parts in equation \eqref{eq:threeparts}:
\begin{align*}
\|D^\alpha( \mu(\cdot,k,0)-1)\|_{W^{s,\tilde p}}&\leq \frac{\tilde c}{|k|^{1-s}}\|f\|_{L^p}\\
&\leq \frac{\tilde c}{4|k|^{1-s}}(\|I\|_{L^p}+\|II\|_{L^p}+\|III\|_{L^p}).
\end{align*}
The norm $\|I\|_{L^p}\leq\|q\|_{W^{n,p}}$ is independent of $k$ and already accounted for in inequality \eqref{eq:mudiff}. For the second norm we use the induction hypothesis with 
\[\|II\|_{L^p}=\|qD^\alpha(\mu-1)\|_{L^p}\leq \|q\|_{L^2}\|D^\alpha(\mu-1)\|_{L^{\tilde p}}\leq c\sum_{m=0}^{n-1}\|q\|_{W^{m,p}}\|q\|_{L^2}^{n-m}.\]
Now we show the third norm decreases faster in $k$ then the other terms, and therefore it will not contribute to the asymptotic behavior. Choosing $2/\tilde p<s_0<1$,
\begin{align*}
\|III\|&\leq \sum_{\beta:0<\beta< \alpha} \binom{\alpha}{\beta}\|D^{\beta} q D^{\alpha-\beta} (\mu-1)\|_{L^p}\\
&\leq \sum_{\beta:0<\beta<\alpha} \binom{\alpha}{\beta}\|D^{\beta} q\|_{L^p}\| D^{\alpha-\beta} (\mu-1)\|_{W^{s_0,\tilde p}}\\
&\leq \frac{c}{|k|^{1-s_0}} \sum_{\beta:0<\beta< \alpha} \binom{\alpha}{\beta}\|D^{\beta} q\|_{L^p}\| q\|_{L^p}^{n}.
\end{align*}
Choosing $k$ large enough finishes the result.
\end{proof}
In equation \eqref{eq:dbar-k-mu}, we will set  $s(k):={\bf t}(k)/(4\pi \bar k)$.
\begin{lem}
If $q\in W^{n,p}_\rho(\mathbb R^2)$ with $\rho>2/p'$ and $1<p<2$ has no exceptional points then $|k|^ns(k)\in L^{r}(|k|>k_0)$ for all $r>\tilde p'$ and large enough $k_0$. Additionally, $s(k)$ is continuous on $\mathbb C\setminus\{0\}$, and $s(k)\in L^2(\mathbb C)$.
\label{lem:sk-decay}
\end{lem}

\begin{proof}
First note that $W^{n,p}_\rho(\mathbb R^2)\subset W^{n,1}\cap W^{n,p}(\mathbb R^2)$ when $\rho>1$. 
Denote the Fourier transform by $\mathcal F(g)(k)=\int e_{-k}(x)g(x)\,dx$. Then we may write 
\[\bar k^{-1}{\bf t}(k) = \bar k^{-1}\mathcal F(q)(k)+\bar k^{-1}\int e_{-k}(x)q(x)(\mu(x,k)-1)\,dx.\]
 By the Hausdorff-Young inequality and the differentiability of $q\in W^{n,1}\cap W^{n,p}(\mathbb R^2)$, the first term satisfies $|k|^{n+1}\bar k^{-1}\mathcal F(q)(k)\in L^{p'}\cap L^\infty(\mathbb R^2)$.
For the second term, we may integrate by parts because $e_{-k}(x)\in W^{n,\infty}(\mathbb R^2)$ for any $n$ and the product $q(\mu-1)\in W^{n,1}(\mathbb R^2)$ by Lemma \ref{lem:mu-diff-forward}. Therefore we have
\begin{align*}
(i\bar k)^n\int e_{-k}(x)q(x)(\mu(x,k)-1)dx&=\int (-1)^n [\bar\partial^n e_{-k}(x)]q(x)(\mu(x,k)-1)dx\\
&=\int e_{-k}(x)\left(\sum_{j=0}^n \binom{n}{j}\bar\partial_x^j(\mu-1)\bar\partial_x^{n-j}q\right)dx.
\end{align*}
Using the decay estimate from Lemma \ref{lem:mu-diff-forward} with $s_0>2/\tilde p$, we get for $k$ large,
\begin{align*}
|k|^n |\bar k^{-1} (t(k)-\mathcal F(q)(k))|&\leq |k|^{-1}\left( \sum_{j=0}^n \binom{n}{j}\|\bar\partial_z^j(\mu-1)\|_{W^{s_0,\tilde p}}\|\bar\partial_z^{n-j} q\|_{L^1}\right)\\
&\leq \frac{c}{|k|^{2-s_0}}\left(\sum_{m=0}^n\|q\|_{W^{m,p}} \|q\|_{L^2}^{n-m}\right)\|q\|_{W^{n,1}}.
\end{align*}
With this $|k|^{s_0-2}$ decay, we get the second term is in $L^r(|k|>k_0)$ if $(2-s_0)r>2$ which is $r>\frac{2}{2-2/\tilde p}=\tilde p'$. Continuity of ${\bf t}(k)$ follows in the same way as that for conductivity type potentials. Nachman proves this in Theorem 4 of \cite{nachman}. Using the continuity and the small $k$ behavior of ${\bf t}$ from theorem \ref{thm:tk-small} implies $\bar k^{-1}{\bf t}(k)\in L^2(|k|<\epsilon)$ for small enough $\epsilon$. 
\end{proof}
 
In the next two lemmas, we only need the results for equation \eqref{eq:dbar-k-mu} without time included. However, under the flow for the Novikov-Veselov equation, the scattering transform for later times $t$ becomes $e^{it(\bar k^3+k^3)}{\bf t}(k)$. This extra phase does not change the $L^p$ space properties of ${\bf t}(k)$, but proving the results in this generality will allow their use in studying the Novikov-Veselov equation at times $t\neq 0$. We therefore define the exponent $\phi(x,k,t)=(kx+\bar k\bar x)+t(k^3+\bar k^3)$. The evolved $\bar\partial_k$ equation for $\mu(x,k,t)$ then can be written as
\begin{equation}
\begin{cases}
\bar\partial_k\mu(x,k,t)=e^{i\phi}s(k)\overline{\mu(x,k,t)}\\
\mu(x,\cdot,t)-1\in L^r(\mathbb C)
\end{cases}
\label{eq:dbar-k-with-t}
\end{equation}
for some $r>2$.  The operator we must study is then $T_x g(k)= \bar\partial_k^{-1}(e^{i\phi}s(k)\overline{g})(k)$. We would then have that $(\mu-1)$ solves the integral equation $(\mu-1)=T_x 1 + T_x(\mu-1)$. Inverting the operator would then yield $(\mu-1)=[I-T_x]^{-1}T_x 1$. The fact that the operator $T_x$ is compact on $L^p$ for $p>2$ can be found in the preprint to Nachman's 1996 paper Lemma 4.2 \cite{nachmanpreprint}. We reproduce it here for the reader's convenience. 

\begin{lem}
If $s(k)\in L^2$ then the operator $T=\bar\partial_{k}^{-1}( s(k) \bar{\,\cdot\,})$ is compact on $L^p$ for all $2<p<\infty$.
\label{lem:Tcompact}
\end{lem}
\begin{proof}
We will prove the result for the dual operator $s(k)\bar\partial^{-1}_k$ on $L^{q}$ for $1<q<2$. First, we have that the operator is bounded on $L^p$. Take $f\in L^q(\mathbb C)$, then by the Hardy-Littlewood-Sobolev inequality with $1/\tilde q=1/q-1/2$
\begin{equation}
\|s(k)\bar\partial_k^{-1}f(\cdot))\|_{L^q(\mathbb C)} \leq \|s(\cdot)\|\|\bar\partial_k^{-1}f(\cdot)\|_{L^{\tilde q}} \leq c \|s\|_{L^2}\|f\|_{L^q}.
\label{eq:Tcompact}
\end{equation}
Now we assume that $s(k)$ is continuous with compact support. We first have $\partial \bar\partial^{-1} f\in L^q$ from the theory of Calderon-Zygmund operators, and using this
\[ \|\nabla s \bar\partial^{-1}_k f \|_{L^q}\leq \|\nabla s\|_{L^2} \|\bar\partial^{-1}_k f\|_{L^{\tilde q}} + \|s\|_{L^\infty} \|\nabla \bar\partial^{-1}_kf\|_{L^q} \leq c\|f\|_{L^q}.\]
Thus we have $\|s \bar\partial^{-1}_k f\|_{W^{1,q}}\leq c\|f\|_{L^q}$, and since $s$ has compact support, we can use Rellich-Kondrachov compactness to show $s\bar\partial^{-1}$ is a compact operator on $L^q$. For $s$ a general function in $L^2$, we may approximate any $T$ by these compact operators and use inequality \eqref{eq:Tcompact} to show $T$ is compact on $L^q$.
\end{proof}

\begin{lem}
If $s(k) \in L^{r}\cap L^2(\mathbb C)$ for some $1<r<2$ then the operator $I-T_x$ is invertible on $L^{\tilde r}(\mathbb C)$ and there is a unique solution of equation \eqref{eq:dbar-k-with-t} with $\mu(x,\cdot,t)-1\in L^{\tilde r}(\mathbb C)$.
\end{lem}
\begin{proof}
From Lemma \ref{lem:Tcompact}, the operator $T_x$ is compact on $L^{r_1}$ for all $2<r_1<\infty$ so $I-T_x$ is Fredholm. Assume $h\in L^{r_1}(\mathbb C)$ solves $(I-T_x)h=0$, then $\bar\partial_k h=e^{i\phi}s(k)\bar h$. Using Brown and Uhlmann's Liouville theorem with the coefficient $e^{i\phi}s(k)\in L^2(\mathbb C)$ we get $h\equiv 0$. Thus $I-T_x$ is invertible on $L^{r_1}(\mathbb C)$.

To construct a solution, we note that formally
\[\mu(x,k,t)-1=[I-T_x]^{-1}T_x 1.\]
We have $T_x 1 = \bar\partial^{-1}[e^{i\phi}s(k)]$ and using the Hardy-Littlewood-Sobolev inequality $T_x 1 \in L^{\tilde r}(\mathbb C)$. Thus $\mu=1+[I-T_x]^{-1}T_x 1$ solves \eqref{eq:dbar-k-with-t}. 
\end{proof}
Equation \eqref{eq:dbar-k-with-t} is conjugate linear, so we will not prove differentiability of $\mu(x,k,t)$ in the $x$ variable by looking at the $\bar\partial_x$ and $\partial_x$ derivatives. Instead we will take real derivatives in the $x=(x_1,x_2)$ variables written as 
\[D^\alpha_x = \left(\frac{\partial}{\partial x_1}\right)^{\alpha_1}\left(\frac{\partial}{\partial x_2}\right)^{\alpha_2}\]  
where $\alpha =(\alpha_1,\alpha_2)$ is a multi-index.
\begin{lem}
If $q\in W^{n,p}_\rho(\mathbb R^2)$ with $1<p<2$, $\rho>1$ is critical or subcritical then the unique solution $\mu(x,\cdot,t)-1\in L^r(\mathbb C)$ for $p'<r<\infty$, of equation \eqref{eq:dbar-k-with-t} is $\alpha$  times differentiable in $x$ and $m$ times differentiable in $t$ for $(3m+|\alpha|)\leq n$. Additionally, the derivatives of the map $(x,t)\to\mu(x,\cdot,t)\in L^r(\mathbb C)$ satisfy $\partial_t^m D^{\alpha}_{x}\mu(x,\cdot,t)\in L^r(\mathbb C)$. The derivatives are given by
\[\partial_t^m D^{\alpha}_{x} \mu(x,k,t)=[I-T_x]^{-1}\bar\partial^{-1}[s(k)f(x,k,t)] \]
where
\[f(x,k,t)=\partial_t^mD^{\alpha}_{x}[e^{i\phi}\overline{\mu(x,k,t)}]-e^{i\phi}\partial_t^mD^\alpha_x \overline{\mu(x,k,t)},\]
and $\bar\partial^{-1}[s(k)f(x,k,t)]\in L^r(\mathbb C)$.
\label{lem:mu-diff}
\end{lem}
\begin{proof}

We illustrate the case $\alpha=(1,0)$, $m=0$. In the following, let $h\in \mathbb R$ and therefore $x+h=(x_1+h)+ix_2$. The function \[D_h \mu(x+h,k,t)=\frac{\mu(x+ h,k,t)-\mu(x,k,t)}{h}\]
 is in $L^r(\mathbb R^2)$ and satisfies the equation
\[\bar\partial_k D_h\mu(x,k,t)=s(k)\left(\overline{\mu(x+h,k,t)}D_h e^{i\phi}+e^{i\phi}\overline{D_h\mu(x,k,t)}\right).\]
We have $s(k)D_h e^{i\phi}\to i(k+\bar k) s(k)e^{i\phi}\in L^2(\mathbb C)$ because $|k|s(k)\in L^2$, $\mu(x+h,k,t)-1\to \mu(x,k,t)-1\in L^r(\mathbb R^2)$ by continuity, and by the $L^r$ continuity of the operators $T_x$ and $[I-T_x]^{-1}$ we have
\[\frac{\partial}{\partial x_1} \mu(x,k,t)=[I-T_x]^{-1}\bar\partial^{-1}[i(k+\bar k)e^{i\phi}s(k)\overline{\mu(x,k,t)-1}+i(k+\bar k)e^{i\phi}s(k)].\]
Here we have $i(k+\bar k)e^{i\phi}s(k)\in L^{2r/(r+2)}(\mathbb R^2)$ by Lemma \ref{lem:sk-decay}, and by the same lemma, $s(k)\in L^2(\mathbb C)$ so $s(k)(\mu(x,k,t)-1)\in L^{2r/(r+2)}(\mathbb R^2)$. The Hardy-Littlewood-Sobolev inequality then shows 
\[\bar \partial^{-1}[i(k+\bar k)e^{i\phi}s(k)\overline{\mu(x,k,t)-1}+i(k+\bar k)e^{i\phi}s(k)]\in L^r(\mathbb C).\]

Derivatives in $t$ follow in the same manner except with factors of $(k^3+\bar k^3)$ pulled down from the exponential.
\end{proof}

\begin{lem}
Suppose that $q\in W^{n,p}_\rho(\mathbb R^2)$ for $1<p<2$ and $\rho>2/p'$. If $\mu$ solves the equation \eqref{eq:dbar-k-with-t} in $L^r(\mathbb C)$ for some $r> p'$, then $\mu$ admits the large-$k$ expansion
\[\mu(x,k,t)= 1+\sum_{j=1}^{n}\frac{a_j(x,t)}{k^{j}}+ o\left(|k|^{-n}\right)\]
for fixed $x$ and $t$.
Moreover, we may take $\alpha$ spatial derivatives and $m$ time derivatives with $(|\alpha|+3m)\leq n$ to get 
\[\partial_t^m D^{\alpha}_{x}(\mu(x,k,t)-1)=\sum_{j=1}^{n-|\alpha|-3m}\frac{\partial_t^m D^{\alpha}_{x} a_j(x,t)}{k^{j}}+o\left(|k|^{-n+|\alpha|+3m}\right).\]
\label{lem:large-mu-exp}
\end{lem}
\begin{proof}
Note that
\[\mu(x,k,t)=1+\frac{1}{\pi}\int_{\mathbb C}\frac{1}{k-\kappa}e^{i\phi(x,k,t)}s(\kappa)\overline{\mu(x,\kappa,t)}d\kappa.\]
Expanding $(k-\kappa)^{-1}$ as a geometric series, we have
\[
a_j(x,t)=\frac{1}{\pi}\int_{\mathbb C} \kappa^{j-1} e^{i\phi(x,\kappa,t)}s(\kappa)\overline{\mu(x,\kappa,t)}d\kappa
\]
with  remainder
\[
R_n(x,k,t)=\frac{1}{\pi}k^{-n}\int \frac{1}{k-\kappa}\kappa^{n}e^{i\phi(x,\kappa,t)}s(\kappa)\overline{\mu(x,\kappa,t)}d\kappa.
\]
We look at $\Omega_1=\{\kappa:|\kappa|\leq 1\}$ and $\Omega_2=\mathbb C\setminus \Omega_1$ separately. In $\Omega_1$, $s(k)$ is in $L^2(\mathbb C)$, and with $\mu(x,\cdot,t)\in L^r(\Omega_1)$  for $p'<r<\infty$, the integral over this region decreases like $|k|^{-1}$ for all $n$. 

In $\Omega_2$, we use Lemma \ref{lem:dbar-estimates}. We have that $k^ns(k)$ is in $L^{r_1}(|k|>1)$ for all $r_1>\tilde p'$ by Lemma \ref{lem:sk-decay}. 
We also have $D^{\alpha}(\mu(x,\cdot,t)-1)\in L^{r}(\mathbb C)$ for $p'<r<\infty$. 
Combining these two results shows that the product is in $L^{p_1}\cap L^{p_2}(\mathbb C)$ for $\tilde p'r/(\tilde p'+r)<p_1<2<p_2\leq r$, so $R_n(x,k,t)=o(|k|^{-n})$.

To show decay of the derivatives, we look at the difference quotients. As in Lemma \ref{lem:mu-diff} we will take derivatives with respect to $x_1$. Let $h$ be a real number, then we have
\begin{align*}
D_h a_j(x,t)=&\frac{1}{\pi}\int_{\mathbb C} \kappa^{j-1} e^{i\phi(x,\kappa,t)}\frac{e^{i(\kappa h+\bar \kappa h)}-1}{h}s(\kappa)\overline{\mu(x+h,\kappa,t)}d\kappa\\
&+\frac{1}{\pi}\int_{\mathbb C} \kappa^{j-1} e^{i\phi(x,\kappa,t)}s(\kappa)\overline{D_h\mu(x,\kappa,t)}d\kappa\\
&=I+II.
\end{align*}
From Lemma \ref{lem:mu-diff} we have $\mu(x+h,\cdot,t)\to \mu(x,\cdot,t)\in L^r(\mathbb C)$ and $D_h \mu(x,\cdot,t)\to \partial_{x_1} \mu(x,\cdot,t)\in L^r(\mathbb C)$. Also, by Lemma \ref{lem:sk-decay} 
\[\kappa^{j-1} s(\kappa)\frac{e^{i(\kappa h +\bar \kappa h)}-1}{h} \to \kappa^{j-1}i(\kappa+\bar \kappa) s(k)\in L^{r'}\]
when $j\leq n$. Altogether then, we have convergence of the partial derivative. 

Using the same methods, we get may get the derivatives for $R_n(x,k,t)$, but because we get extra factors of $\kappa$ ($3$ for every time derivative, $1$ for every space derivative) we may only take the expansion to order $(n-|\alpha|-3m)$ so that the derivatives of $R_{n-|\alpha|-3m}(x,k,t)$ will converge.
\end{proof}

\begin{proof}[Proof of Theorem \ref{thm:inv-scat}]
The solutions $\mu(x,k)$ from the equations \eqref{eq:dbar-k-mu} and \eqref{eq:dbar-z} are the same, so we may plug-in the large $k$ expansion of $\mu$ into equation \eqref{eq:dbar-z}. By our assumptions on $q$ and Lemma \ref{lem:large-mu-exp} we may take $x$ derivatives of $\mu(x,\cdot)-1\in L^{r}(\mathbb C)$.  The large-$k$ expansion of $\mu$ gives us $|\bar\partial\partial \mu((x,k)|=o(1)$ as $|k|\to\infty$ since $q$ has two derivatives. Thus, using equation \eqref{eq:dbar-z} and the formula for $a_1(x)$, we write
\begin{align*}
q(x)\mu(x,k)=&4\bar\partial(\partial+ik) \mu(x,k) \\
=&\frac{4i}{\pi}\bar\partial\int_{\mathbb C} e_{-\kappa}(x)s(\kappa)\overline{\mu(x,\kappa)}\,d\kappa+o(1).
\end{align*}
Taking the limit as $|k|\to \infty$, and using the fact that $\mu(x,k)\to 1$ point-wise as $|k|\to \infty$, we have
\[q(x)= \frac{4i}{\pi}\bar\partial\int_{\mathbb C} e_{-\kappa}(x)s(\kappa)\overline{\mu(x,\kappa)}\,d\kappa.\]
\end{proof}

\bibliographystyle{plain}

\begin{thebibliography}{10}

\bibitem{astala}
K.~Astala, T.~Iwaniec, and G.~Martin.
\newblock {\em Elliptic partial differential equations and quasi-conformal
  mappings in the plane}, volume~48 of {\em Princeton Mathematical Series}.
\newblock Princeton University Press, Princeton, NJ, 2009.

\bibitem{beals}
R.~Beals and R.~R. Coifman.
\newblock Linear spectral problems, nonlinear equations and the
  {$\overline\partial$}-method.
\newblock {\em Inverse Problems}, 5(2):87--130, 1989.

\bibitem{boitietal}
M.~Boiti, J.~Leon, M.~Manna, and F.~Pempinelli.
\newblock On a spectral transform of a {KDV-like} equation related to the
  {Schrodinger} operator in the plane.
\newblock {\em Inverse Problems}, 3(1):25, 1987.

\bibitem{brown_uhlmann}
R.~Brown and G.~Uhlmann.
\newblock Uniqueness in the inverse conductivity problem for nonsmooth
  conductivities in two dimensions.
\newblock {\em Communications in Partial Differential Equations},
  22(5-6):1009--1027, 1997.

\bibitem{croke}
R.~Croke, J.~Mueller, M.~Music, P.~Perry, S.~Siltanen, and A.~Stahel.
\newblock The {Novikov--Veselov} equation: Theory and computation.
\newblock {\em Inverse Problems}, 2013.

\bibitem{faddeev}
L.~D. Faddeev.
\newblock Increasing solutions of the {Schr\"odinger} equation.
\newblock {\em Soviet Physics Doklady}, volume~10:1033--1035, 1966.

\bibitem{grinevich-novikov}
P.~G. Grinevich and R.~G. Novikov.
\newblock Faddeev eigenfunctions for point potentials in two dimensions.
\newblock {\em Phys. Lett. A}, 376(12-13):1102--1106, 2012.

\bibitem{lms}
M.~Lassas, J.~L. Mueller, and S.~Siltanen.
\newblock Mapping properties of the nonlinear Fourier transform in dimension
  two.
\newblock {\em Communications in Partial Differential Equations},
  32(4):591--610, 2007.

\bibitem{lmss}
M.~Lassas, J.L. Mueller, S.~Siltanen, and A.~Stahel.
\newblock The {Novikov--Veselov} equation and the inverse scattering method,
  {Part I}: Analysis.
\newblock {\em Physica D: Nonlinear Phenomena}, 241(16):1322 -- 1335, 2012.

\bibitem{murata}
M.~Murata.
\newblock Structure of positive solutions to {$(-\Delta+V)u=0$} in {${\bf
  R}^n$}.
\newblock {\em Duke Math. J.}, 53(4):869--943, 1986.

\bibitem{exceptional_circle}
M.~Music, P.~Perry, and S.~Siltanen.
\newblock Exceptional circles of radial potentials.
\newblock {\em Inverse Problems}, 29(4):045004, 25, 2013.

\bibitem{nachmanpreprint}
A.~I. Nachman.
\newblock Global uniqueness for a two-dimensional inverse boundary value
  problem.
\newblock {\em University of Rochester, Dept. of Mathematics Preprint Series},
  19, 1993.

\bibitem{nachman}
A.~I. Nachman.
\newblock Global uniqueness for a two-dimensional inverse boundary value
  problem.
\newblock {\em Ann. of Math. (2)}, 143(1):71--96, 1996.

\bibitem{perry}
P.~Perry.
\newblock Miura maps and inverse scattering for the {Novikov Veselov} equation.
\newblock {\em Preprint}, 2013.
\newblock arxiv:1201.2385, submitted to {\em Analysis and Partial Differential
  Equations}.

\bibitem{siltanen}
S.~Siltanen.
\newblock Electrical impedance tomography and {Faddeev}'s {Green} functions.
\newblock {\em Ann. Acad. Sci. Fenn. Mathematica Dissertationes}, 121, 1999.

\bibitem{tsai2}
T.-Y. Tsai.
\newblock The associated evolution equations of the Sch\"odinger operator in the plane.
\newblock {em Inverse Problems}, 10 (1994), no. 6, 1419-1432.

\bibitem{tsai}
T.-Y. Tsai.
\newblock The {S}chr\"odinger operator in the plane.
\newblock {\em Inverse Problems}, 9(6):763--787, 1993.

\end{thebibliography}



\end{document}